\documentclass{IEEEtran}
\usepackage[T1]{fontenc}
\usepackage[utf8]{luainputenc}
\usepackage{color}
\usepackage{float}
\usepackage{amsmath}
\usepackage{amsthm}
\usepackage{amssymb}
\usepackage{stmaryrd}
\usepackage{esint}
\usepackage{enumitem}
\PassOptionsToPackage{normalem}{ulem}
\usepackage{ulem}

\DeclareMathOperator*{\argmin}{arg\,min}
\usepackage{graphicx}
\makeatletter

\floatstyle{ruled}
\newfloat{algorithm}{tbp}{loa}
\providecommand{\algorithmname}{Algorithm}
\floatname{algorithm}{\protect\algorithmname}

\theoremstyle{plain}
\newtheorem{thm}{\protect\theoremname}
\theoremstyle{plain}

\theoremstyle{definition}

\theoremstyle{remark}

\theoremstyle{plain}

\theoremstyle{plain}
\newtheorem{lem}[thm]{\protect\lemmaname}
\ifx\proof\undefined
\newenvironment{proof}[1][\protect\proofname]{\par
\normalfont\topsep6\p@\@plus6\p@\relax
\trivlist
\itemindent\parindent
\item[\hskip\labelsep\scshape #1]\ignorespaces
}{%
\endtrivlist\@endpefalse
}
\providecommand{\proofname}{Proof}
\fi
\theoremstyle{remark}

\makeatother

\usepackage[english]{babel}
\providecommand{\claimname}{Claim}
\providecommand{\definitionname}{Definition}
\providecommand{\lemmaname}{Lemma}
\providecommand{\propositionname}{Proposition}
\providecommand{\remarkname}{Remark}
\providecommand{\theoremname}{Theorem}
\providecommand{\corollaryname}{Corollary}
\newcommand{\x}{\textbf{x}}

\newcommand{\z}{\textbf{z}}
\newcommand{\bb}{\beta}
\newcommand{\w}{\textbf{w}}

\newcommand{\EX}{\mathbb{E}}

\newcommand{\R}{\mathbb{R}}

\newcommand{\G}{\mathcal{G}}
\newcommand{\E}{\mathcal{E}}
\newcommand{\V}{\mathcal{V}}
\newcommand{\N}{\mathcal{N}}

\newcommand{\LA}{\mathcal{L}}
\newcommand{\pip}{\bar{\pi}}
\newcommand{\tl}{\tilde{\lambda}}

\newcommand{\rr}{\rho}
\usepackage{latexsym}

\begin{document}

\title{Linearly Convergent Asynchronous Distributed ADMM via Markov Sampling }

\author{S. M. Shah and K. Avrachenkov}

\maketitle

\begin{abstract}
We consider the consensual distributed optimization problem and propose an asynchronous version of the Alternating Direction Method of Multipliers (ADMM) algorithm to solve it.  The 'asynchronous' part here refers to the fact that only one node/processor is updated at each iteration of the algorithm. The selection of the node to be updated is decided by simulating a Markov chain. The proposed algorithm is shown to have a linear convergence property in expectation for the class of functions which are strongly convex and continuously differentiable. 
\end{abstract}

\section{Introduction }

In this paper we consider the following optimization problem :
\begin{equation}\label{intro-pron}
\min_{x \in \mathbb{R}^n} \sum_{i=1}^N f_i(x)
\end{equation}
where $f^i \,:\,\mathbb{R}^n \to \mathbb{R}$. A lot of problems of interest in data science, network systems and autonomous control can be formulated in the above form. The most prevalent example comes from machine learning where the above formulation  is used in Empirical Risk Minimization (see \cite{bot}). It involves approximating the following problem
\begin{equation}\label{prob1}
\min_{x \in \Omega} \big\{ \mathbb{E}_{\xi}\, [F(x, \xi) ] \doteq \int F(x,\cdot)\, dP(\cdot) \big\},
\end{equation}
where $P$ is the probability law of the random variable $\xi$. While it may be desirable to minimize (\ref{prob1}), such a goal is untenable when one does not have access to the law $P$ or when one cannot draw from an infinite population sample set. A practical approach is to instead seek the solution of a problem that involves an estimate of the the expectation in (\ref{prob1}) giving rise to (\ref{intro-pron}). In that case one has $f^i(x) = F(x,\xi_i)$ for some realization $\xi_i$ of $\xi$.

Optimization problems of the form (\ref{intro-pron}) usually have a large $N$, which makes finding a possible solution by first order methods a computationally intensive and time consuming task.  Also, because of high dimensions, they are generally beyond the capability of second order methods due to extremely high iteration complexity. A possible remedy is to use a distributed approach. This generally involves distributing the data sets among $N$ machines, with each machine having its own estimate of the variable $x$. Such data distribution may necessitated by the fact that all of it cannot be stored in one machine.Also,  data may also be distributed across different machines because of actual physical constraints (for instance, data may be collected in a decentralized fashion by the nodes of a network which has communication constraints). This has led to an increasing interest in developing distributed optimization algorithms to solve (\ref{intro-pron}).\\

\textbf{Contributions :} In this paper we propose an asynchronous variant of ADMM to solve problems of the form (\ref{intro-pron}). Asynchronous  here means that only a single node of the network is randomly updated at any instant of the algorithm while the rest of the nodes do not perform any computation. This helps in overcoming a major dis-advantage of synchronous algorithms wherein all the nodes have to update before proceeding to the next step. This is usually a deterrent to distributed computing since the different nodes tend to have different processing speed, may suffer network delays or failure which will harm the progress of the algorithm. A remedy for these problems is to consider asynchronous algorithms. Our contributions here are as follows : We formulate (\ref{intro-pron}) in a form that is amenable to be solved by ADMM in a distributed fashion. The proposed algorithm to solve it uses asynchronous updates where the selection of the node to be updated is decided by simulating a Markov chain. Another advantage here is that since all consensus ADMM algorithms use local information, the agent updating it's information at a given time will use the information of at least one neighbour which has been updated in the most recent step, thereby ensuring progress. The second contribution here is that the proposed algorithm is shown to have a \textit{linear convergence} property in expectation for the class of functions which are strongly convex and continuously differentiable. This, to the best of our knowledge, is the first work that considers \textit{fully distributed }(not using central server architectures) ADMM algorithm with Markov sampling and establishes a linear convergence.\\

\textbf{Related Literature :} We briefly discuss the relevant ADMM literature here, for a survey of the general ADMM algorithm, we refer the reader to \cite{Boyd}. The literature on distributed optimization is vast building upon the works of \cite{Ber} for the unconstrained case and \cite{Ned} for constrained case. \cite{Ber} was one of the earliest works to address the issue of achieving a consensus solution to an optimization problem in a network of computational agents. In \cite{Ned}, the same problem was considered subject to constraints on the  solution. The works \cite{Ram} and \cite{Kun} extended this framework and studied different variants and extensions (asynchrony, noisy links, varying communcation graphs, etc.). The non-convex version was considered in \cite{Bia}. There has been a lot research on solving distributed optimization problems with ADMM in the last decade. Some of the works include Chapter 6 \cite{Boyd}, \cite{Wei1}, \cite{Sch}, \cite{Shi}, \cite{Zhang} among others. The major works in the context of asynchronous ADMM include \cite{Wei2}, \cite{Bianchi}, \cite{Mota}, \cite{Kwok}, \cite{Qing} and \cite{Hui}. Incremental Markov updates have been previously considered for first order methods in \cite{Ram2} and \cite{Joh}. We leverage these ideas for the ADMM algorithm because of the ease of implementation of Markov updates for distributed optimization.\\

The organization of the paper is as follows. In Section II, we give the problem formulation and propose an asynchronous distributed ADMM algorithm to solve it. In Section III,  we present the convergence analysis of the proposed algorithm and show that it has a linear convergence rate for a certain class of functions. In Section IV, some numerical experiments are given and the concluding remarks are given in Section 5.

\section{Background}

In this section we give the details of the problem considered in the paper and the proposed  asynchronous algorithm to solve it. We consider (\ref{intro-pron}) problem in a distributed setting. We first quickly review the details of the distributed model we assume. Also, we let $n=1$ throughout the paper for simplicity. \textit{The algorithm and  its analysis for $n>1$ is the same using the Kronecker notation.} 

\subsection{Distributed model}

We consider a network of $N$ agents indexed by $1, ..., N$ and  associate with each agent $i$, the function $f_i$. For future use, let $f:\mathbb{R} \to\mathbb{R}$ denote
\begin{equation}
f(\cdot) := \frac{1}{N}\sum_{i=1}^Nf_i(\cdot).
\end{equation}

We assume the communication network is modelled by a static undirected graph
$\mathcal{G=}\{\mathcal{V},\mathcal{E}\}$ where $\mathcal{V}=\{1,...,N\}$
is the node set and $\mathcal{E}$ is the edge set of links $(i,j)$ indicating that agent $i$ and  $j$  can exchange information. We let $\mathcal{A}\subset\mathcal{V}\times\mathcal{V}$ denote the set of arcs, so that $|\mathcal{A}|=2|\E|$. We make the following assumption on the graph :\\

\noindent \textbf{Assumption 1} : The undirected graph $\mathcal{G}$ is connected.\\
 
To solve (\ref{intro-pron}) in a decentralized fashion, we  construct the transition matrix $P$ of a Markov chain (MC) , using only local information, such that the following property is satisfied :
$$[P]_{ij} > 0 \Longleftrightarrow (i,j)\in\mathcal{E}.$$
 In addition, we also want $P$ to satisfy :\\
\begin{enumerate}
\item[(N1)] {[}\textit{Irreducibility and aperiodicity}{]} The underlying graph is irreducible, i.e., there is a directed path from any  node to any other node, and  aperiodic, i.e., the g.c.d.\ of lengths of all paths from a node to itself is one. It is known that the choice of node in this definition is immaterial. This property can be guaranteed, e.g.,  by making $p_{ii}>0$ for some $i$.
\end{enumerate}
Given a connected graph $\G$, one of the ways to generate a transition matrix that satisfies the above properties is given by the Metropolis-Hastings scheme.\\

We have the following result (Theorem 5, \cite{Chat},  Lemma 4.1 \cite{Ram2}). 
\begin{lem}\label{geo-lem}
Suppose (N1) is satisfied. Then, we have
\begin{equation*}
\lim_k P^k =  \textbf{1}_N \pi^T
\end{equation*}  
where $\pi$ is a column vector denoting the stationary distribution of $P$ and $\textbf{1}_N$ is vector with all entries equal to $1$.  Furthermore, the convergence rate is geometric so that
\begin{equation}\label{geo}
|p_{ij}^k - \pi(j)| \leq b \,\gamma^k  \,\,\,\forall i,\, j,
\end{equation}
where $\gamma \in (0,1)$ and $b>0$ are some constants.
\end{lem} 

In the distributed setting which we consider here, (\ref{intro-pron}) can be reformulated as :
\begin{equation}\label{Dist}
\begin{aligned}
\text{minimize} & \qquad \sum_{i=1}^{N} f_i(x_i), \\
\text{  subject to } & \qquad x_i =x_j \,\,\,\,\,\,\forall \,  i=1,..,N \text{ and }\, j \in \N(i)
\end{aligned}
\end{equation}
where $x^i$ is the estimate of the primal variable $x$ at node $i$. Our main aim here is to solve (\ref{Dist}) in an asynchronous distributed fashion via the ADMM method. In the next subsection we give the details of the ADMM algorithm.

\subsection{ADMM Algorithm}
The standard undistributed ADMM solves the following problem 
\begin{equation}\label{mainprob}
\begin{aligned}
\text{minimize} & \qquad  f(x) + g(z), \\
\text{subject to}  & \qquad Ax+Bz=c,
\end{aligned}
\end{equation}
where variables $x\in \mathbb{R}^n$ , $z\in \R^m$, $A\in \R^{ p\times n}$, $B\in \R^{p \times m}$, and $c \in  \R ^p$.
To solve the problem, we consider the augmented Lagrangian which is defined as 
$$
\LA_{\rho}(x,z,y) =  f(x) + g(z) + y^T (Ax + Bz - c) + \frac{\rho}{2} \|Ax + Bz - c\|^2.
$$
Starting from some initial vector $(x_0,z_0,y_0)$, ADMM consist of the following updates :
\begin{equation}\label{x}
x^{k+1} = \argmin_x \LA_{\rho}(x,z^k,y^k),
\end{equation}
\begin{equation}\label{z}
z^{k+1} = \argmin_z \LA_{\rho}(x^{k+1},z,y^k),
\end{equation}
\begin{equation}\label{y}
y^{k+1} = y^k +\rho (Ax^{k+1} + bz^{k+1} -c).
\end{equation}

\subsection{Asynchronous ADMM }
The augmented Lagrangian of problem (\ref{Dist}) can be written as :
\begin{multline*}
\LA_{\rr} ( \x,\lambda) = \sum_{i=1}^{N}f_i (x_i) + \sum_{i=1}^N \lambda_{ij} (\sum_{j \in \N(i)} x_i -x_j) +\\ \frac{\rr}{2}\sum_{i=1}^N \sum_{j \in \N(i)} (x_i -x_j) ^2,
\end{multline*}
where $\lambda_{ij}$ is the Lagrange multiplier associated with the constraint $x_i -x_j$. We note that the third term in the RHS of the above equation is not separable which  a makes a distributed implementation impossible in this formulation. There have been many re-formulations of (\ref{Dist}) in the consensus ADMM literature to make a distributed implementation possible. We use the formulation presented in Section 3.4,  Section 3.4, \cite{Bert}. Similar formulations have been previously used in \cite{Makh},  \cite{Wei2} and \cite{Qing}. To do this we introduce auxiliary variables $z_{ij}$ with each edge between any two nodes $i$ and $j$. The reformulation can be written as : 
\begin{equation}\label{Dist-2}
\begin{aligned}
\text{minimize} & \qquad \sum_{i=1}^{N} f_i(x_i), \\
\text{subject to}   & \qquad x_i=x_j =z_{ij}, \,\,\forall \,( i,j) \in \mathcal{A} ,\\
\end{aligned}
\end{equation}
The augmented Lagrangian for the above problem can be written as
\begin{multline*}
\LA_\rr (\x,\z, \lambda ) = \sum_{i=1}^{N} f_i(x_i) +   \\\sum_{i=1}^{N} \sum_{j \in \N(i) }\Big( \lambda_{ij} (x_i - z_{ij}) + \tilde{\lambda}_{ij} (x_j - z_{ij} )  \Big)  \\+ \frac{\rr}{2} \sum_{i=1}^{N} \sum_{j \in \N(i) } \Big( \|x_i - z_{ij}\|^2 +  \|x_j - z_{ij} \|^2 \Big),
\end{multline*}
where $\lambda_{ij}$ is the Lagrange multiplier associated with the constraint $x_i = z_{ij}$ for every $i \in \V$ and $\tl_{ij}$ the Lagrange multiplier associated with the constraint $x_j =z_{ij}$ for $j\in \N(i)$. The ADMM algorithm will take the following form :
\begin{multline}\label{1}
x_i(k+1)  := \argmin_{x} \Big\{ f_i(x) + \sum_{j \in \N(i) }   (\lambda_{ij}(k)+\tl_{ji}(k) ) x \\+ \frac{\rr}{2}  \sum_{j \in \N(i) }  \big( \|x  - z_{ij}(k)\|^2 +   \|x  - z_{ji}(k)\|^2  \big)  \Big\}, 
\end{multline}
\begin{multline}\label{2}
 z_{ij}(k+1)   := \argmin_{z } \Big\{  ( \lambda_{ij}(k) + \tl_{ij} ) z  \\+\frac{\rr}{2} \big( \|x_i(k+1) - z \|^2 +\|x_j(k+1) - z \|^2   \big)  \Big\} 
\end{multline}
and 
\begin{align}
\lambda_{ij} (k+1) &= \lambda_{ij} (k) + \rr \Big(x_i (k+1) -z_{ij} (k+1) \Big) \label{3}  \\
\tl_{ij} (k+1) &= \tl_{ij} (k) + \rr \Big(x_j (k+1) -z_{ij} (k+1) \Big)\label{3'}
\end{align}
We note that the update for $z_{ij}$ (eq. \ref{2}) has a closed form solution given by :
\begin{multline}\label{z-ij}
z_{ij}(k+1) =  \frac{1}{2 \rr} (\lambda_{ij}(k) + \tl_{ij}(k) ) + \frac{1}{2} (x_i(k+1) + x_j(k+1))  \\  i=1,..,N,\,j\in\N(i),
\end{multline}
Adding equations (\ref{3}) and (\ref{3'}), we have 
\begin{multline*}
\lambda_{ij} (k+1)+\tl_{ij} (k+1) = \lambda_{ij} (k) +\tl_{ij} (k)  +\\ \rr \Big(x_i (k+1) -z_{ij} (k+1)  + x_j (k+1) -z_{ij} (k+1) \Big)
\end{multline*}
Using (\ref{z-ij}) to substitute for $z_{ij}(k+1)$ in the above, we have
\begin{equation}\label{eol}
\lambda_{ij}(k+1) = - \tl_{ij}(k+1)
\end{equation}
So if $\lambda_{ij}(0)=-  \tl_{ij}(0) $, we have $\lambda_{ij}(k) = - \tl_{ij}(k)$ for all $k$. The update of $z_{ij}$ (see (\ref{z-ij})) can be simplified to
\begin{equation}\label{zzzz}
z_{ij}(k+1) = \frac{1}{2} \Big( x_i(k+1) +x_j(k+1) \Big)
\end{equation}
Also, if we use (\ref{zzzz}) in (\ref{3}), we have 
\begin{equation}\label{anotherequality}
\lambda_{ij}(k+1) =\lambda_{ij}(k) + \frac{c}{2}\big(  x_i(k+1) - x_j(k+1) \big).
\end{equation}
So,  if we set $\lambda_{ij}(0)=-\lambda_{ji}(0)$, we have
\begin{equation}\label{anotherequality-1}
\lambda_{ij}(k)=-\lambda_{ji}(k)
\end{equation}
for all $k$. Let $\mu_i(k) := 2 \sum_{j \in \N(i)} \lambda_{ij}(k) $. The update of $x_i$ (see (\ref{1})) can then be written in a form amenable to distributed implementation :
\begin{multline}\label{final-x-upadate}
x_i(k+1)  := \argmin_{x} \Big\{ f_i(x) + \mu_i^T(k)  x \\+ \rr \sum_{j \in \N(i) }  \|x  - \frac{x_i(k) + x_j(k)}{2}\|^2   \Big\},
\end{multline}
where we have used the fact that $\lambda_{ji}(k) \stackrel{\text{(\ref{eol})}}{=}-\tl_{ji}(k+1)$ and $\tl_{ji}(k+1) \stackrel{\text{(\ref{anotherequality-1})}}{=} -\lambda_{ji}(k)  $ in the second term of RHS of (\ref{1}) and (\ref{zzzz}) in the third term. 
\begin{algorithm}[H]
\textbf{Input :}
Graph $\G$, Transition Matrix $P$, Functions $\{f_i\}_{i=1}^N$, Time Step $\rho$.\\

\textbf{Initial Conditions :}

Initialize $\mathbf{x}_{0} \in \R^N$ and $\mu_0\in \R^N$. Let $i_0 \in \mathcal{V}$ be the initial state of the Markov chain.\\

\textbf{Algorithm :}\\

For $k=0,1,2.....$ do : \\
\begin{enumerate}

\item Let $\xi_k$ denote the state of the MC at time $k$. \\

\item  Set :
\begin{multline}\label{x-update}
x_{\xi_k}(k+1)  := \argmin_{x} \Big\{ f_i(x) + \mu_{\xi_k}^T(k)  x \\+\rr  \sum_{j \in \N(\xi_k ) }  \|x  - \frac{x_{\xi_k}(k) + x_j(k)}{2}\|^2   \Big\}, 
\end{multline}
For $ i \in \mathcal{V}/\{\xi_k\}$, set $x_i(k+1) = x_i(k) $.\\

\item  Set :
\begin{align}\label{w-update}
\mu_{\xi_k}(k+1) &= \mu_{\xi_k}(k) + \rr \sum_{j \in \N(\xi_k)} \big(  x_i(k+1) - x_j(k+1) \big).
\end{align}
For $ i \in \mathcal{V}/\{\xi_k\}$, set $\mu_i(k+1) = \mu_i(k) $.\\

\end{enumerate}

\textbf{Output :} $\{x_i\}_{t=0}^{\infty}$ for any $i \in \V$

\caption{Asynchronous ADMM }
\end{algorithm}
Also, using (\ref{anotherequality}), the update for $\mu_j$  can be written as 
\begin{equation}\label{final-mu-upadate}
\mu_i(k+1) = \mu_i(k) + \rr \sum_{i \in \N(i)} \big(  x_i(k+1) - x_j(k+1) \big).
\end{equation}
The final algorithm uses updates (\ref{final-x-upadate}) and (\ref{final-mu-upadate}). The pseudo-code for the asynchronous version is given in Algorithm 1.

\section{Convergence Analysis :} 

In this section we aim to show that the proposed algorithm has a linear convergence rate in expectation when the following condition is satisfied :\\

\noindent{\textbf{Assumption 2 :} }Each function $f_i$ is strongly convex and $L$-smooth in $\mathbb{R}$. \footnote{Hence, so is $f$.} By strong convexity, we mean that for any $x$ and $y$ in the domain of $f_i(\cdot)$, we have ,\begin{equation}\label{stron-convex}
\langle \nabla f_i(x)- \nabla f_i(y), x-y \rangle \geq \nu \|x -y \|^2
\end{equation}
with $\nu>0$. By smoothness, we mean that the gradient is $L$-Lipschitz continuous, i.e.
\begin{equation}\label{contdiff}
\| \nabla f_i(x) -\nabla f_i(y)\| \leq L \|x-y\|.
\end{equation}
We can take $L>0$ to be independent of $i$ without any loss of generality. Let $x \in \mathbb{R}^{N}$ be a vector concatenating all $x_i$ and $z$ be a vector concatenating all $z_{ij}$. Let $I_{n}$ denote the $n\times n$ identity matrix. We can rewrite the constraints in (\ref{Dist-2}) in the following matrix form :
\begin{equation}\label{mainprob-1}
\begin{aligned}
\text{minimize} & \qquad  f(x) + g(z), \\
\text{subject to}  & \qquad Ax+Bz=c,
\end{aligned}
\end{equation}
where 
\[A=
\begin{bmatrix}
    A_1     \\
    A_2     
\end{bmatrix}
\,\, \text{  and  } \,\,B=
\begin{bmatrix}
    -I_{2|\E| }       \\
    -I_{2|\E| } 
\end{bmatrix} 
\]
with $A_1, A_2 \in \mathbb{R}^{ 2|\E| \times N}$. The entries of the matrix $A$ are decided as follows : If $(i,j) \in \mathcal{A}$ and $z_{ij}$ is the $q$'th entry of $z$, then the $(q,i)$ entry of $A_1$  is one and $(q,j)$ entry of $A_2$ is one. If $\lambda$ denotes the concatenated Lagrange multipliers, we have from the KKT conditions :
\begin{equation}\label{KKT-new}
\nabla f(x(k+1)) +A^T \lambda(k) +\rr A^T(Ax(k+1) + Bz(k))=0
\end{equation}
and the dual update
\begin{equation}\label{lagg}
\lambda(k+1) = \lambda(k) + \rr(Ax(k+1) +Bz(k+1) ).
\end{equation}
We remark here (\ref{lagg}) is the same update as (\ref{3})-(\ref{3'}). To be more precise, the concatenated variable $\lambda(k)$ can be written as $ \lambda(k) = [\beta(k) ,  \gamma(k) ]$, where $\bb(k) \in \mathbb{R}^{|\E|}$ with $\bb_{ij}(k) := \lambda_{ij}(k)$ and $\gamma(k) \in \mathbb{R}^{|\E|}$ with $\gamma(k) := \tl_{ij}(k)$. We note that from (\ref{eol}), that $\lambda(k) = [\beta(k), -\beta(k)]$. Multiplying (\ref{lagg}) with $A^T$ and adding to (\ref{KKT-new}) we have, 
\begin{equation}\label{z-mega}
\nabla f(x(k+1)) +A^T\lambda(k+1) +\rr A^TB(z(k) - z(k+1))=0
\end{equation}
We let $I_+ := A_1^T +A_2^T$ denote the unoriented incidence matrix and $I_- := A_1^T -A_2^T$ denote the oriented incidence matrix. Then, (\ref{z-mega}) can be written as,
\begin{equation}\label{z-mega1}
\nabla f(x(k+1)) +I_- \beta(k) + \rr I_+ (z(k+1) - z(k))=0.
\end{equation}
Also, from (\ref{zzzz}) and (\ref{anotherequality}), we have 
\begin{equation}\label{proofz}
z(k) = \frac{1}{2} I_+^Tx(k)
\end{equation}
and 
\begin{equation}\label{proofb}
\beta(k+1) = \bb (k) + \frac{\rr}{2} I_-^Tx(k+1) 
\end{equation}
We use (\ref{z-mega1}), (\ref{proofz}) and (\ref{proofb}) to obtain our main result. Also, the KKT conditions for (\ref{mainprob-1}) give,
\begin{align}
\nabla f(x^*)+ I_- \bb^*  &= 0\label{KKT-1}\\
I_-^Tx^* &=0\label{KKT-2}\\
 z^* &= \frac{1}{2} I_+^Tx^*,\label{KKT-3}
\end{align}
where $(x^*, z^*)$ is the unique primal optimal solution and the uniqueness follows from the strong convexity of $f(\cdot)$. We note here that if $\bb(0)$ lies in the column space of $I_-$, then $\bb(k)$ also does so from (\ref{proofb}). We can also assume $\bb^*$ lies in the column space of $I_-$ without any loss of generality (see eq. (11), \cite{Wei1}). We use these facts later in the proof of our main result. We let $G \in \mathbb{R}^{4|\E| \times 4|\E|}$ denote the matrix 
\[G=
\begin{bmatrix}
    \rr I_{2|\E|}  & 0_{2|\E|}   \\
   0_{2|\E|} &  \frac{1}{\rr} I_{2|\E|}  
\end{bmatrix}
\]
Let $\pi_{\text{min}} := \min_{ \{i =1,..,N\}} \pi(i)$,  $\pi_{\text{max}} := \max_{\{i=1,..,N\}} \pi(i)$ and $\sigma_{\text{max}}(A)$, $\sigma_{\text{min}}(A)$ respectively denote the largest and smallest \textit{non-zero} singular values of the matrix $A$.

\begin{thm}\label{manthrm}
Let $w(k)$ denote the concatenated vector $ w(k):= [ z(k) \,\, \bb(k) ]$ and $ w^* := (z^*\,\, \bb^*)$, where $z^*$ and $\bb^*$ are the unique optimal primal-dual pair. We have for some $k'>0$ satisfying 
\begin{equation}\label{k'}
k' \geq \frac{1}{\ln \gamma}\,\ln \Big(\frac{(1+c)\pi_{\text{min}} - \pi_{\text{max}}}{b(2+c)} \Big)
\end{equation}
and all $k>k'$, 
$$ \EX[ \| w(k+1) - \w^*\|_G^2  ] \leq  (1- \alpha_{k'})^{k-k'}  \EX[ \|w(k') - \w^* \|_G^2], $$
where  
\begin{multline}\label{c}
c :=  \min \Big\{  \frac{(\kappa -1)\sigma^2_{\text{min}}(I_-) }{ \kappa  \sigma^2_{\text{max}}(I_+) }, \\  \frac{\nu}{\Big( \frac{\kappa L^2}{\rho \sigma^2_{\text{min}}(I_-) } + \frac{\rho}{4} \sigma^2_{\text{max}}(I_+)\Big)}  \Big\},
\end{multline}
for any $\kappa >0$ and 
$$\alpha_{k'} = \pi_{\text{min}}  - b \gamma^{k'}- \frac{1}{1+c} \Big\{ \pi_{\text{max}}+ b \gamma^{k'}  \Big\} > 0.$$
\end{thm}
\begin{proof}
Let $i_0$ denote the initial state of the Markov chain. We note at time $k$, the probability of updating the estimate $w_i(k)$ at the $i$'th node to $\hat{w}_i(k+1) := (\hat{z}_i(k+1) \,\, \hat{\bb}_i(k+1)) $, with $\hat{z}_i(k+1)$ given by (\ref{proofz}) and $\hat{\bb}(k+1)$ given by (\ref{proofb}),  is $p^k_{i_0i}$. The probability of $w_i(k+1) = w_i(k)$ is $1-p^k_{i_0i}$ (i.e. the event of $i$ not being updated). Thus, we have
\begin{multline}\label{11}
\EX[\| z_i(k+1) - z^*\|^2 ] = \Big(1-p^k_{i_0i} \Big) \EX[ \| z_i(k) -z^* \|^2]  \\ +p^k_{i_0i} \EX [ \|\hat{z}_i(k+1) - z^* \|^2]
\end{multline}
and 
\begin{multline}\label{22}
\EX[\| \bb_i(k+1) - \bb^*\|^2 ] = \Big(1-p^k_{i_0i} \Big) \EX[ \| \bb_i(k) - \bb^* \|^2]  \\ +p^k_{i_0i} \EX [ \|\hat{\bb}_i(k+1) - \bb^* \|^2]
\end{multline}
for all $i$. Set $\pi_{\text{min}} := \min_{ \{i =1,..,N\}} \pi(i)$ and  $\pi_{\text{max}} := \max_{\{i=1,..,N\}} \pi(i)$. Set $\bar{\delta}(k) := \pi_{\text{max}}  + b \gamma^k$ and $\underset{\bar{}}{\delta}(k) := \pi_{\text{min}} - b \gamma^k$. From Lemma \ref{geo-lem} (Section 2.1), we have 
\begin{equation}\label{piminmax}
\underset{\bar{}}{\delta}(k)  \leq p^k_{i_0i} \leq \bar{\delta}(k)
\end{equation}
Multiplying (\ref{11}) by $\rr$ and (\ref{22}) by $1/\rr$, and summing both equations along the index $i$ and using (\ref{piminmax}), we have
\begin{multline*}
\EX\big[ \sum_{i=1}^N \big\{ \rr\|z_i(k+1) - z^*\|^2 + \frac{1}{\rho}\|\bb_i(k+1) - \bb^*\|^2 \big\} \big]\\
\leq   (1-\underset{\bar{}}{\delta}(k))\EX \big[ \sum_{i=1}^N \big\{ \rr \| z_i(k+1) - z^*\|^2 + \frac{1}{\rho}\|\bb_i(k) - \bb^*\|^2 \big\} \big]\\
+  \bar{\delta}(k) \EX \big[ \sum_{i=1}^N \big\{ \rr \| \hat{z}_i(k+1) - z^*\|^2 + \frac{1}{\rho}\|\hat{\bb}_i(k) - \bb^*\|^2 \big\} \big]
\end{multline*}
which gives using the definition of $G$-norm
\begin{multline}\label{main eq-MC}
\EX[ \| w(k+1) - \w^*\|_G^2  ] \leq  (1-\underset{\bar{}}{\delta}(k)) \EX[ \| w(k) - \w^* \|_G^2 ]  + \\ \bar{\delta}(k) \EX[ \|\hat{w}(k+1) - \w^* \|_G^2]
\end{multline}
where $\w^* := (w^*,...,w^*)$. Associated  with $\hat{w}(k+1)$, we have the vector $\hat{x}(k+1)$ which satisfies (see (\ref{KKT-new})),
\begin{equation*}
\nabla f(\hat{x}(k+1)) +A^T \lambda(k) +\rr A^T(A\hat{x}(k+1) + Bz(k))=0
\end{equation*}
Also, the dual update $\hat{\lambda}(k+1) := [\hat{\beta}(k+1), -\hat{\beta}(k+1)]$ satisfies
\begin{equation*}\label{lag}
\hat{\lambda}(k+1) = \lambda(k) + \rr(A\hat{x}(k+1) +B\hat{z}(k+1) ).
\end{equation*}
Proceeding the same way as for deriving (\ref{z-mega1})-(\ref{proofb}), we get 
\begin{align}
\nabla f(\hat{x}(k+1))   +  I_- \hat{\bb}(k+1)   &= \rr I_+ (z(k) -\hat{z}(k+1 ))  \label{1-ex'}\\
\frac{\rr}{2} I_- \hat{x}(k+1)   &= \hat{\bb}(k+1) - \bb(k) \label{2-ex'}\\
\frac{1}{2} I_+ \hat{x}(k+1)  &= \hat{z}(k+1) \label{3-ex'}
\end{align}
We next prove the following claim whose proof identical to Theorem 1, \cite{Wei1}. We prove it here for the sake of completeness.\\ 

\noindent \textit{Claim :} $ \|\hat{w}(k+1) - \w^* \|_G^2 \leq \frac{1}{1+c} \|w(k+1) - \w^* \|_G^2 $, with $c$ as in (\ref{c}).\\

\noindent \textit{Proof.} Subtracting (\ref{KKT-1})-(\ref{KKT-3}) from (\ref{1-ex'})-(\ref{3-ex'}) respectively , we have the following set of equations
\begin{align}
\nabla f(\hat{x}(k+1)) - \nabla f (x^*) &= \rr I_+ (z(k) -\hat{z}(k+1 )) \nonumber\\
& \,\,\,\,\,  -I_-( \hat{\bb}(k+1) - \bb^*) \label{1-ex}\\
\frac{\rr}{2} I_- (\hat{x}(k+1) -x^*)  &= \hat{\bb}(k+1) - \bb(k) \label{2-ex}\\
\frac{1}{2} I_+ (\hat{x}(k+1) -x^*)  &= \hat{z}(k+1) - z^*\label{3-ex}
\end{align}
We have from (\ref{stron-convex}),
\begin{equation}\label{st}
\nu \| \hat{x}(k+1) -x^* \|^2 \leq \langle \nabla f(\hat{x}(k+1)- \nabla f(x^*),  \hat{x}(k+1)-x^* \rangle
\end{equation} 
Using (\ref{1-ex}), the RHS of the above can be written as :
\begin{multline*}
 \langle \nabla f(\hat{x}(k+1)- \nabla f(x^*),  \hat{x}(k+1)-x^* \rangle =\\\langle \rr I_+ (z(k) -z(k+1) ) -I_-( \hat{\bb}(k+1) - \bb^* )   , \hat{x}(k+1)-x^* \rangle 
\end{multline*}
so that 
\begin{multline*}
 \langle \nabla f(\hat{x}(k+1)- \nabla f(x^*),  \hat{x}(k+1)-x^* \rangle =\\ \langle  z(k) -\hat{z}(k+1 ),  \rr I^T_+(\hat{x}(k+1)-x^*) \rangle\\  - \langle  \hat{\bb}(k+1) - \bb^*    , I^T_- (\hat{x}(k+1)-x^*) \rangle 
\end{multline*}
Using (\ref{2-ex}) and (\ref{3-ex}) in the above, we get
\begin{multline*}
 \langle \nabla f(\hat{x}(k+1)- \nabla f(x^*),  \hat{x}(k+1)-x^* \rangle =\\ 2 \rr \langle  z(k) -\hat{z}(k+1 ), \hat{z}(k+1) - z^* \rangle\\  +\frac{2}{\rr} \langle\bb(k)-\hat{\bb}(k+1), \hat{\bb}(k+1) -\bb^*      \rangle =\\
 2 \langle  w(k) - \hat{w}(k+1), \hat{w}(k+1) -w^* \rangle_G
\end{multline*}
Using the quality $2 \langle  w(k) - \hat{w}(k+1), \hat{w}(k+1) -w^* \rangle_G=  \|w(k) -w^* \|_G- \|\hat{w}(k+1) -w^* \|_G- \|w(k) -\hat{w}(k+1) \|_G$ , we have
\begin{multline*}
 \langle \nabla f(\hat{x}(k+1)- \nabla f(x^*),  \hat{x}(k+1)-x^* \rangle =\\ \|w(k) -w^* \|_G- \|\hat{w}(k+1) -w^* \|_G- \|w(k) -\hat{w}(k+1) \|_G
\end{multline*}
and using (\ref{st}), 
\begin{multline}\label{equivalclaim}
 \nu \| \hat{x}(k+1) -x^* \|^2  \leq  \|w(k) -w^* \|^2_G \\ - \|\hat{w}(k+1) -w^* \|^2_G - \|w(k) -\hat{w}(k+1) \|^2_G
\end{multline}
To prove the claim, we just need to show that 
\begin{equation*}
\|w(k) -\hat{w}(k+1) \|^2_G + \nu \| \hat{x}(k+1) -x^* \|^2  \geq c \|\hat{w}(k+1) -w^* \|^2_G,
\end{equation*}
since adding the above to (\ref{equivalclaim}) gives the required inequality. We note that the above inequality is equivalent to 
\begin{multline}\label{equivalclaim-1}
\rr \|\hat{z}(k+1) - z(k) \|^2 + \frac{1}{\rr}\| \hat{\bb}(k+1) - \bb(k) \|^2 + \nu \| \hat{x}(k+1) -x^* \|^2 \\  \geq c \rr \|\hat{z}(k+1) - z^* \|^2 + \frac{c}{\rr}\| \hat{\bb}(k+1) - \bb^* \|^2
\end{multline}
To upper bound $\|\hat{z}(k+1) - z^*\|^2$, we use (\ref{3-ex}),
\begin{multline}\label{33}
\|\hat{z}(k+1) - z^*\|^2 = \frac{1}{4}\| I_+ (\hat{x}(k+1) -x^*)\|^2\\
\leq  \frac{1}{4} \sigma^2_{\text{max}}(I_+)\|\hat{x}(k+1) -x^* \|^2
\end{multline}
To upper bound $\| \hat{\bb}(k+1) - \bb(k) \|^2 $, we first consider for any $\kappa>0$,
\begin{multline*}
\|\rr I_+^T( z(k) -\hat{z}(k+1)) \|^2 +(\kappa-1)\|   \nabla f(\hat{x}(k+1))- \nabla f(x^*)\| \leq \\ \rr^2 \sigma^2_{\text{max}}(I_+)\|( z(k) -\hat{z}(k+1)) \|^2 +(\kappa-1)L^2 \|\hat{x}(k+1) -x^* \|^2
\end{multline*}
We have used the continuous differentiability of $f(\cdot)$ in the above (see (\ref{contdiff})). Using the inequality $\|a+b\|^2 + (\kappa -1)\|a\|^2 \geq (1-\frac{1}{\kappa})\|b \|^2 $ in the above we have,
\begin{multline*}
\|\rr I_+^T( z(k) -\hat{z}(k+1)) \|^2 +(\kappa-1)\|   \nabla f(\hat{x}(k+1))- \nabla f(x^*)\|^2\\
 \geq \Big(1-\frac{1}{\kappa} \Big)\| I_- (\hat{\bb}(k+1) -\bb^*)\|^2,
\end{multline*}
using (\ref{1-ex}). As mentioned previously $\beta_k,\, \bb^*$ lie in column space of $I_-$ so that $ \| I_- (\hat{\bb}(k+1) -\bb^*)\|^2 \geq \sigma^2_{\text{min}}(I_-)  \|  (\hat{\bb}(k+1) -\bb^*)\|$. Using this fact, we have from the last two inequalities 
\begin{multline*}
\rr^2 \sigma^2_{\text{max}}(I_+)\|( z(k) -\hat{z}(k+1)) \|^2 +(\kappa-1)L^2 \|\hat{x}(k+1) -x^* \|^2
\geq \\ 
 \Big(1-\frac{1}{\kappa} \Big)\sigma^2_{\text{min}}(I_-)  \|  \hat{\bb}(k+1) -\bb^*\|^2
\end{multline*} 
which gives
\begin{multline}\label{36}
\frac{\rr \kappa  \sigma^2_{\text{max}}(I_+) }{(\kappa -1)\sigma^2_{\text{min}}(I_-) }\|\hat{z}(k+1) - z^* \|^2 \\ + \frac{\kappa L^2}{\rr \sigma^2_{\text{min}}(I_-) } \| \hat{x}(k+1) -x(k)\|^2 
 \geq \frac{1}{\rr}\| \hat{\bb} (k+1) - \bb^*\|^2 
\end{multline}
Adding (\ref{33}) multiplied by $\rho$ to  (\ref{36}), we have,
\begin{multline}
\frac{ \rr \kappa  \sigma^2_{\text{max}}(I_+) }{(\kappa -1)\sigma^2_{\text{min}}(I_-) }\|\hat{z}(k+1) -z(k) \|^2 +\\
 \Big( \frac{\kappa L^2}{\rr \sigma^2_{\text{min}}(I_-) } + \frac{\rr}{4} \sigma^2_{\text{max}}(I_+)\Big) \| \hat{x}(k+1) -x(k)\|^2 \\
 \geq \rr \|\hat{z}(k+1) -z^* \|^2 + \frac{1}{\rr}\| \hat{\bb} (k+1) - \bb^*\|^2.
\end{multline}
Using the value of $c$ specified in (\ref{c}), we get 
\begin{multline*}
\rr \|\hat{z}(k+1) -z(k) \|^2 + \nu \|\hat{x}(k+1)-x^* \|^2\\ 
\geq c \|\hat{w}(k+1) -w^* \|_G^2.
\end{multline*}
which proves (\ref{equivalclaim-1}) and hence the claim.\\

We continue with the proof of the theorem. Using the statement of the previous claim in (\ref{main eq-MC}), we have 
\begin{align*}
\EX[ \| w(k+1) - \w^*\|_G^2  ] &\leq  (1-\underset{\bar{}}{\delta}(k)) \EX[ \| w(k) - \w^* \|_G^2 ]  + \\
& \qquad  \qquad \bar{\delta}(k)\,\frac{1}{1+c} \EX[ \|w(k) - \w^* \|_G^2]\\
& \leq (1-\underset{\bar{}}{\delta}(k) +\bar{\delta}(k)\,\frac{1}{1+c}) \times \\ 
& \qquad \EX[ \|w(k) - \w^* \|_G^2]
\end{align*}
Recalling the definition of $\underset{\bar{}}{\delta}(k)$ and $\bar{\delta}(k)$, we have
$$
 \bar{\delta}(k)\,\frac{1}{1+c}-\underset{\bar{}}{\delta}(k)\leq 0
$$
for all $k>k'$, such that
\begin{equation}
k' \geq \frac{1}{\ln \gamma}\,\ln \Big(\frac{(1+c)\pi_{\text{min}} - \pi_{\text{max}}}{b(2+c)} \Big)
\end{equation}
which completes the proof of the theorem.
\end{proof}

We next prove the linear convergence of $x(k)$. We fix $k'$ and use the index $l=k-k'$ for simplicity. Also, set $\eta := 1- \alpha_{k'}$ (we drop the subscript $k'$ for ease of notation). We have the following inequality for $x(k)$, similar to (\ref{main eq-MC}), 
\begin{multline}\label{main eq-MC111}
\EX[ \| x(l+1) - \x^*\|^2  ] \leq  (1-\underset{\bar{}}{\delta}(k)) \EX[ \| x(l) - \x^* \|^2 ]  + \\ \bar{\delta}(k) \EX[ \|\hat{x}(l+1) - \x^* \|^2],
\end{multline}
where $\x^* := (x^*,...,x^*)$. From (\ref{equivalclaim}), we have 
 \begin{equation*}
\EX[ \|\hat{x}(l+1) - \x^* \|^2] \leq\frac{1}{\nu} \EX[ \|w(l) - \w^* \|^2]
\end{equation*}
Using the above inequality in (\ref{main eq-MC111}), we get
\begin{multline*}
\EX[ \| x(l+1) - \x^*\|^2  ] \leq  (1-\underset{\bar{}}{\delta}(k)) \EX[ \| x(l) - \x^* \|^2 ]  + \\ \frac{\bar{\delta}(k) }{\nu} \EX[ \|w(l) - \w^* \|^2],
\end{multline*}
From the linear convergence of $w$, we have 
\begin{multline*}
\EX[ \| x(l+1) - \x^*\|^2  ] \leq  (1-\underset{\bar{}}{\delta}(k)) \EX[ \| x(l) - \x^* \|^2 ]  + \\ \frac{\bar{\delta}(k) \eta^l }{\nu} \EX[ \|w(0) - \w^* \|^2],
\end{multline*}
For a large enough $k'$, we have $0<\pip_{\text{min}} \leq \pi_{\text{min}}-b \gamma^{k'} <  \pi_{\text{max}}+b\gamma^{k'} \leq \pip_{\text{max}} <1$ for some $\pip_{\text{max}}$,  $\pip_{\text{min}}$. Then we have, 
\begin{multline*}
\EX[ \| x(l+1) - \x^*\|^2  ] \leq  (1-\pip_{\text{min}}) \EX[ \| x(l) - \x^* \|^2 ]  + \\ \frac{\pip_{\text{max}} \, \eta^l }{\nu} \EX[ \|w(0) - \w^* \|^2].
\end{multline*}
Iterating the above inequality, we get 
\begin{multline*}
\EX[ \| x(l+1) - \x^*\|^2  ] \leq  (1-\pip_{\text{min}})^{l+1} \EX[ \| x(0) - \x^* \|^2 ]  + \\ 
+ \frac{(1-\pip_{\text{min}})^{l} \pip_{\text{max}} \,  }{\nu} \EX[ \|w(0) - \w^* \|^2] +....\\
...+\frac{(1-\pip_{\text{min}})\pip_{\text{max}}  \, \eta^{l-1} }{\nu} \EX[ \|w(0) - \w^* \|^2]\\ + \frac{\pip_{\text{max}} \, \eta^l }{\nu} \EX[ \|w(0) - \w^* \|^2].
\end{multline*}
Assume $1-\pip_{\text{min}} \leq \eta$. The proof is the same for $\eta \leq 1-\pip_{\text{min}}$ with their roles interchanged. Set $M : = \frac{\pip_{\text{max}} \, \EX[ \|w(0) - \w^* \|^2] }{\nu}$. We have  
\begin{multline*}
\EX[ \| x(l+1) - \x^*\|^2  ] \leq  (1-\pip_{\text{min}})^{l+1} \EX[ \| x(0) - \x^* \|^2 ]  + \\ 
+ M\eta^l \Big\{  1+ \frac{1-\pip_{\text{min}}}{\eta} +.. +\Big(\frac{1- \pip_{\text{min}}}{\eta}\Big)^l \Big\}
\end{multline*}
so that 
\begin{multline*}
\EX[ \| x(l+1) - \x^*\|^2  ] \leq  (1-\pip_{\text{min}})^{l+1} \EX[ \| x(0) - \x^* \|^2 ]  + \\ 
+ \frac{M}{1-\frac{(1-\pip_{\text{min}})}{\eta}}\eta^l
\end{multline*}
Let $W>0$ be a constant such that
\begin{equation*}
\Big(\frac{1-\pip_{\text{min}}}{\eta}\Big)^l (1-\pip_{\text{min}})
+ \frac{M}{\big(1-\frac{(1-\pip_{\text{min}})}{\eta} \big)\EX[ \| x(0) - \x^* \|^2} \leq W
\end{equation*}
Then,
\begin{equation*}
\EX[ \| x(l+1) - \x^*\|^2  ] \leq  W\eta^l  \,\EX[ \| x(0) - \x^* \|^2
\end{equation*}
which proves the linear convergence of $x(\cdot)$.\\

Note that $k'$ may be not be well defined in all cases. However, for MC's with uniform steady state distribution, it will be always well defined.
\section{Numerical Experiment}
In this section we evaluate the performance of the algorithm on a simple estimation problem. We consider this tutorial problem since our main concern here is to compare the proposed asynchronous version with the existing synchronous version of ADMM. The distributed estimation problem involves estimating a parameter $x^*$, using noisy measurements performed  at each node. We set the local measurement as $a_i := x^* + N(0,1) \in \mathbb{R}^n$, where $N(0,1)$  represents Gaussian noise  with mean zero and unit variance. The problem can be formulated as
\begin{equation}\label{sim}
\begin{aligned}
\text{minimize} & \qquad \sum_{i=1}^{N} \|x_i-a_i\|^2 , \\
\text{  subject to } & \qquad x_i =x_j \,\,\forall \, i,j.
\end{aligned}
\end{equation}
The Markov chain considered is a simple random walk with states $\V :=\{1,...,N \}$ and transition probabilities given as :   
\begin{equation*}
P_{i,i+1} = P_{i,i-1} = \alpha , \,\,\, i=2,...,N-1.
\end{equation*}
Also, $P_{12}=1-P_{11}=P_{N,N}=1-P_{N,N-1}=\alpha$.  The steady state distribution of such a Markov chain is known to be
\begin{equation}\label{beta}
\pi_i = \frac{\beta^i(1 - \beta)}{1- \beta^{N}},
\end{equation}
where $\beta = \frac{1}{1-\alpha}$.  We consider $n=10$ and $N=10$ nodes. The results have been plotted in Figure 1. 
\begin{figure}
\begin{center}
\includegraphics[width=9cm,height=7.5cm]{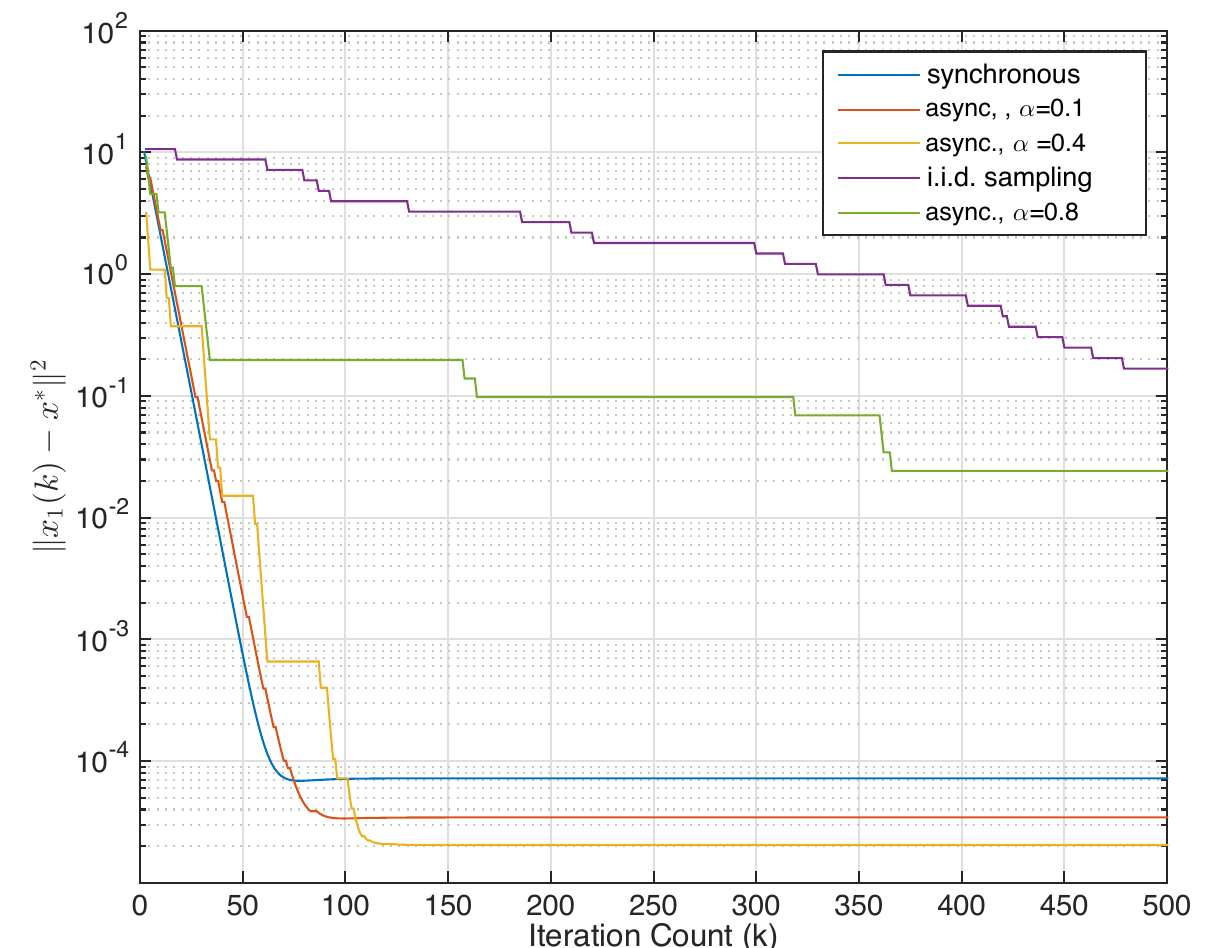}
\end{center}
\caption{Plot of the optimality error vs iteration count }
\end{figure}
Some points to note here are :
\begin{itemize}
\item The most significant observation here is that the asynchronous version with $\alpha=0.1$ performs on par with the synchronous version even though the iteration complexity of the former is quite low. A possible reason for this is the form of the function in (\ref{sim}). If the noise is small, the values $a_i$ fall within a small range and hence the $f_i()$'s can be quite similar.\\

\item As $\alpha $ increases from $\alpha_{}=0.1$ to  $\alpha_{}=0.8$, the performance of the algorithm degrades significantly. A possible explanation here is given by studying the corresponding steady state distributions. For $\alpha= 0.1$, from (\ref{beta}) we have $\pi_{\text{min}}=0.05 $ and   $\pi_{\text{max}}= 0.14$. Hence, all the nodes here have a significant probability of getting activated once the chain has mixed, ensuring progress. Also, we note that the difference $\pi_{\text{max}} - \pi_{\text{min}}$ is small here which makes $k'$ small in (\ref{k'}). For $\alpha = 0.8$, we have $\pi_{i} \leq 0.006,\, i\leq 7 $ and   $\pi_{\text{max}}= 0.8$. Once the chain mixes, there is very low probability of encountering states other than $10$. We note that the algorithm shows a linear convergence rate.\\

\item The i.i.d sampling scheme fares the worst. One possible explanation could be since the nodes are selected at random, while updating there is a low probability of encountering neighbouring nodes that have updated recently. This may explain the periods of no progress prevalent in the staircase pattern of the graph.

\end{itemize}

\bibliographystyle{IEEEtran}
\bibliography{biblo}

\begin{thebibliography}{10}
\providecommand{\url}[1]{#1}
\csname url@samestyle\endcsname
\providecommand{\newblock}{\relax}
\providecommand{\bibinfo}[2]{#2}
\providecommand{\BIBentrySTDinterwordspacing}{\spaceskip=0pt\relax}
\providecommand{\BIBentryALTinterwordstretchfactor}{4}
\providecommand{\BIBentryALTinterwordspacing}{\spaceskip=\fontdimen2\font plus
\BIBentryALTinterwordstretchfactor\fontdimen3\font minus
  \fontdimen4\font\relax}
\providecommand{\BIBforeignlanguage}[2]{{%
\expandafter\ifx\csname l@#1\endcsname\relax
\typeout{** WARNING: IEEEtran.bst: No hyphenation pattern has been}%
\typeout{** loaded for the language `#1'. Using the pattern for}%
\typeout{** the default language instead.}%
\else
\language=\csname l@#1\endcsname
\fi
#2}}
\providecommand{\BIBdecl}{\relax}
\BIBdecl

\bibitem{bot}
L.~Bottou, F.~E. Curtis, and J.~Nocedal, ``Optimization methods for large-scale
  machine learning,'' \emph{arXiv preprint arXiv:1606.04838}, 2016.

\bibitem{Boyd}
S.~Boyd, N.~Parikh, E.~Chu, B.~Peleato, J.~Eckstein \emph{et~al.},
  ``Distributed optimization and statistical learning via the alternating
  direction method of multipliers,'' \emph{Foundations and
  Trends{\textregistered} in Machine learning}, vol.~3, no.~1, pp. 1--122,
  2011.

\bibitem{Ber}
J.~Tsitsiklis, D.~Bertsekas, and M.~Athans, ``Distributed asynchronous
  deterministic and stochastic gradient optimization algorithms,'' \emph{IEEE
  transactions on automatic control}, vol.~31, no.~9, pp. 803--812, 1986.

\bibitem{Ned}
A.~Nedic, A.~Ozdaglar, and P.~A. Parrilo, ``Constrained consensus and
  optimization in multi-agent networks,'' \emph{IEEE Transactions on Automatic
  Control}, vol.~55, no.~4, pp. 922--938, 2010.

\bibitem{Ram}
S.~S. Ram, A.~Nedi{\'c}, and V.~V. Veeravalli, ``Distributed stochastic
  subgradient projection algorithms for convex optimization,'' \emph{Journal of
  optimization theory and applications}, vol. 147, no.~3, pp. 516--545, 2010.

\bibitem{Kun}
K.~Srivastava and A.~Nedic, ``Distributed asynchronous constrained stochastic
  optimization,'' \emph{IEEE Journal of Selected Topics in Signal Processing},
  vol.~5, no.~4, pp. 772--790, 2011.

\bibitem{Bia}
P.~Bianchi and J.~Jakubowicz, ``Convergence of a multi-agent projected
  stochastic gradient algorithm for non-convex optimization,'' \emph{arXiv
  preprint arXiv:1107.2526}, 2011.

\bibitem{Wei1}
E.~Wei and A.~Ozdaglar, ``Distributed alternating direction method of
  multipliers,'' 2012.

\bibitem{Sch}
I.~D. Schizas, A.~Ribeiro, and G.~B. Giannakis, ``Consensus in ad hoc wsns with
  noisy links—part i: Distributed estimation of deterministic signals,''
  \emph{IEEE Transactions on Signal Processing}, vol.~56, no.~1, pp. 350--364,
  2008.

\bibitem{Shi}
W.~Shi, Q.~Ling, K.~Yuan, G.~Wu, and W.~Yin, ``On the linear convergence of the
  admm in decentralized consensus optimization,'' \emph{IEEE Transactions on
  Signal Processing}, vol.~62, pp. 1750--1761, 2014.

\bibitem{Zhang}
C.~Zhang, H.~Lee, and K.~Shin, ``Efficient distributed linear classification
  algorithms via the alternating direction method of multipliers,'' in
  \emph{Artificial Intelligence and Statistics}, 2012, pp. 1398--1406.

\bibitem{Wei2}
E.~Wei and A.~Ozdaglar, ``On the o (1= k) convergence of asynchronous
  distributed alternating direction method of multipliers,'' in \emph{Global
  conference on signal and information processing (GlobalSIP), 2013
  IEEE}.\hskip 1em plus 0.5em minus 0.4em\relax IEEE, 2013, pp. 551--554.

\bibitem{Bianchi}
F.~Iutzeler, P.~Bianchi, P.~Ciblat, and W.~Hachem, ``Asynchronous distributed
  optimization using a randomized alternating direction method of
  multipliers,'' in \emph{Decision and Control (CDC), 2013 IEEE 52nd Annual
  Conference on}.\hskip 1em plus 0.5em minus 0.4em\relax IEEE, 2013, pp.
  3671--3676.

\bibitem{Mota}
J.~Mota, J.~Xavier, P.~Aguiar, and M.~Puschel, ``D-admm: A
  communication-efficient distributed algorithm for separable optimization,''
  \emph{IEEE Transactions on Signal Processing}, vol.~61, no.~10, pp.
  2718--2723, 2013.

\bibitem{Kwok}
R.~Zhang and J.~Kwok, ``Asynchronous distributed admm for consensus
  optimization,'' in \emph{International Conference on Machine Learning}, 2014,
  pp. 1701--1709.

\bibitem{Qing}
Q.~Ling and A.~Ribeiro, ``Decentralized dynamic optimization through the
  alternating direction method of multipliers,'' \emph{IEEE Transactions on
  Signal Processing}, vol.~5, no.~62, pp. 1185--1197, 2014.

\bibitem{Hui}
T.-H. Chang, M.~Hong, W.-C. Liao, and X.~Wang, ``Asynchronous distributed admm
  for large-scale optimization—part i: algorithm and convergence analysis,''
  \emph{IEEE Transactions on Signal Processing}, vol.~64, no.~12, pp.
  3118--3130, 2016.

\bibitem{Ram2}
S.~S. Ram, A.~Nedi{\'c}, and V.~V. Veeravalli, ``Incremental stochastic
  subgradient algorithms for convex optimization,'' \emph{SIAM Journal on
  Optimization}, vol.~20, no.~2, pp. 691--717, 2009.

\bibitem{Joh}
B.~Johansson, M.~Rabi, and M.~Johansson, ``A randomized incremental subgradient
  method for distributed optimization in networked systems,'' \emph{SIAM
  Journal on Optimization}, vol.~20, no.~3, pp. 1157--1170, 2009.

\bibitem{Chat}
S.~Chatterjee and E.~Seneta, ``Towards consensus: Some convergence theorems on
  repeated averaging,'' \emph{Journal of Applied Probability}, vol.~14, no.~1,
  pp. 89--97, 1977.

\bibitem{Bert}
D.~P. Bertsekas, \emph{Parallel and distributed computation: numerical
  methods}.\hskip 1em plus 0.5em minus 0.4em\relax Prentice hall Englewood
  Cliffs, NJ, 1989, vol.~23.

\bibitem{Makh}
A.~Makhdoumi and A.~Ozdaglar, ``Convergence rate of distributed admm over
  networks,'' \emph{IEEE Transactions on Automatic Control}, vol.~62, no.~10,
  pp. 5082--5095, 2017.

\end{thebibliography}

\end{document}